\documentclass{amsproc}
\usepackage{subfiles}
\usepackage[utf8]{inputenc}
\usepackage[english]{babel}
\usepackage{comment}
\usepackage[alphabetic]{amsrefs}
\usepackage{tikz}
\usepackage{xcolor}
\usepackage{datetime2} 
\usepackage[colorlinks=true,linkcolor=blue,]{hyperref}
\usepackage{tcolorbox}
\usepackage{xstring}

\usepackage{stackengine}
\usepackage{scalerel}
\usepackage{mathtools}
\usepackage{calc}
\usepackage{amsthm}
\usepackage{thmtools}
\usepackage[framemethod=TikZ]{mdframed}
\usepackage{amssymb}
\usepackage{amsfonts}
\usepackage{euscript}
\usepackage{fourier}
\usepackage{dsfont}
\renewenvironment{itemize}
  {\begin{list}{$\triangleright$}{%
  \setlength{\parskip}{0mm}
  \setlength{\topsep}{.1\baselineskip}
  \setlength{\rightmargin}{0mm}
  \setlength{\listparindent}{0mm}
  \setlength{\itemindent}{0mm}
  \setlength{\labelwidth}{3ex}
  \setlength{\itemsep}{.1\baselineskip}
  \setlength{\parsep}{.1\baselineskip}
  \setlength{\partopsep}{0mm}
  \setlength{\labelsep}{1ex}
  \setlength{\leftmargin}{\labelwidth+\labelsep}
  }}{%
\end{list}}

\newtheoremstyle{mio}
     {2\parskip}     
     {2\parskip}     
     {}
     {}
     {\bfseries}
     {}
     {1ex}
     {\llap{\thmnumber{#2}\hskip0.9ex}\thmname{#1}\thmnote{\bfseries{}#3}}

\newcounter{thm}

\tcbset{
  mythm/.style={
    colback=black!5!white,
    colframe=white!50!black,
    extrude right by=2ex,
    extrude left by=6.5ex,
    top=0.7ex,
    bottom=1.5ex,
    right=2ex,
    left=6.5ex,
    boxsep=0pt,
    before skip=\baselineskip,
    after skip=\baselineskip,
  },
}
\theoremstyle{mio}
\newtheorem{theorem}[thm]{Theorem}\tcolorboxenvironment{theorem}{mythm}
\tcolorboxenvironment{corollary}{mythm}
\tcolorboxenvironment{proposition}{mythm}
\newtheorem{lemma}[thm]{Lemma}\tcolorboxenvironment{lemma}{mythm}
\newtheorem{fact}[thm]{Fact}\tcolorboxenvironment{fact}{mythm}
\newtheorem{definition}[thm]{Definition}\tcolorboxenvironment{definition}{mythm}
\tcolorboxenvironment{assumption}{mythm}
\tcolorboxenvironment{void}{mythm}
\newtheorem{notation}[thm]{Notation}\tcolorboxenvironment{notation}{mythm}
\tcolorboxenvironment{note}{mythm}
\newtheorem{remark}[thm]{Remark}\tcolorboxenvironment{remark}{mythm}
\tcolorboxenvironment{definition_theorem}{mythm}

\makeatletter
\providecommand{\proofNameStyle}{\bfseries}
\renewenvironment{proof}[1][\proofname]{\par
  \pushQED{\qed}%
  \normalfont
  \trivlist
  \item[\hskip\labelsep
        \proofNameStyle
    #1\@addpunct{.}]\ignorespaces
}{%
  \popQED\endtrivlist\@endpefalse
}
\makeatother

\renewcommand*{\emph}[1]{%
   \smash{\tikz[baseline]\node[rectangle, fill=teal!25, rounded corners, inner xsep=0.5ex, inner ysep=0.2ex, anchor=base, minimum height = 2.7ex]{\strut #1};}}

\author{Silvia Barbina}
\author{Riccardo Camerlo}
\author{Domenico Zambella}

\address[Silvia Barbina]{Sezione di Matematica, Universit\`{a} di Camerino}
\email[Silvia Barbina]{silvia.barbina@unicam.it}
\address[Riccardo Camerlo]{Dipartimento di Matematica, Universit\`{a} di Genova}\email[Riccardo Camerlo]{riccardo.camerlo@unige.it}
\address[Domenico Zambella]{Dipartimento di Matematica, Universit\`{a} di Torino}
\email[Domenico Zambella]{domenico.zambella@unito.it}

\thanks{Silvia Barbina was partially supported by PRIN2022 \textit{Models, sets and classifications}, prot.\@ 2022TECZJA.
Riccardo Camerlo was partially supported by the MUR excellence department project awarded to the Department of Mathematics of the University of Genoa, CUP D33C23001110001. 
Domenico Zambella was partially supported by PRIN2022 \textit{Logical methods in combinatorics}, prot.\@ 2022BXH4R5}


\begin{document}
\vspace*{-1ex}
\title{Local stability in structures with a standard sort}
\maketitle
\raggedbottom

\begin{abstract}
  Recently, a classical approach to continuous structures has been proposed in \cite{clcl} and~\cite{Z} that extends the class of structures falling under the scope of~\cite{HI} or~\cite{BBHU}.
  These articles introduce the notion of \textit{structures with a standard sort.}
  We discuss local stability in this context.
  We examine three variants of the order property which are prima facie non equivalent.
  For each variant we show that sets externally definable by stable formulas are definable in some appropriate sense.
\end{abstract}

\def\medrel#1{\parbox{6ex}{\hfil $#1$}}
\def\ceq#1#2#3{\parbox[t]{13ex}{$\displaystyle #1$}\medrel{#2}{$\displaystyle #3$}}

\section{Structures with a standard sort}

Continuous logic~\cite{BBHU} has replaced Henson-Iovino logic~\cite{HI} as a formalism to study the model theory of continuous structures. 
One of the first articles on the subject, \cite{BU}, defines local stability within continuous logic.
As noted by Henson, this definition escapes~\cite{HI}'s approach.
This has subsequently been used as a major argument to advocate in favor of \cite{BBHU}'s over~\cite{HI}'s approach.

In this section we recall the basic properties of the \textit{structures with a standard sort\/} in \cite{clcl}, which offers an alternative framework to describe continuous structures using classical logic. 

There is more than one way to define stability in \cite{clcl}'s context.
In Sections~\ref{finitary}, \ref{nonfinitary} and~\ref{epsilonstable} we consider three variants and for each we prove a suitable version of the classical theorem that says that externally definable sets are definable. 
All three proofs follow the same line of reasoning~--~a finitary version of the classical argument as presented e.g. in~\cite{Z?}.

Let \emph{$S$\/} be a Hausdorff compact topological space.
We associate to $S$ a first order structure in a language \emph{${\EuScript L}_{\sf S}$\/} that has a relation symbol for each compact subset $C\subseteq S$ and a function symbol for each continuous function $f:S^n\to S$.
Depending on the context, $C$ and $f$ denote either the symbols of ${\EuScript L}_{\sf S}$ or their interpretation in the structure $S$.
Throughout these notes the letter $C$ always denotes a compact subset of $S$, or a tuple of such sets.

Finally, note that we could allow in ${\EuScript L}_{\sf S}$ relation symbols for all compact subsets of $S^n$, for any $n$.
But this would clutter the notation, therefore we prefer to leave the straightforward generalization to the reader.


We also fix an arbitrary first order language which we denote by \emph{${\EuScript L}_{\sf H}$\/} and call the language of the home sort.

\begin{definition}\label{def_0}
  Let \emph{${\EuScript L}$\/} be a two sorted language. 
  The two sorts are denoted by \emph{${\sf H}$} and \emph{${\sf S}$.} 
  The language ${\EuScript L}$ expands ${\EuScript L}_{\sf H}$ and ${\EuScript L}_{\sf S}$ with symbols of sort ${\sf H}^n\times{\sf S}^m\to {\sf S}$.
  An \emph{${\EuScript L}$-structure\/} is a structure of signature ${\EuScript L}$ that interprets these symbols with equicontinuous functions (i.e.\@ uniformly continuous w.r.t.\@ the variables in ${\sf H}$).

  For a given $S$ a \emph{standard structure\/} is a two-sorted ${\EuScript L}$-structure of the form $\langle M,S\rangle$, where $M$ is any structure of signature ${\EuScript L}_{\sf H}$.
  Standard structures are denoted by the domain of their home sort.
\end{definition}

We denote by ${\EuScript F}$ the set of ${\EuScript L}$-formulas constructed inductively from atomic formulas of the form (i) and (ii) below using the Boolean connectives $\wedge$, $\vee$; the quantifiers $\forall\raisebox{1.1ex}{\scaleto{\sf H}{.8ex}\kern-.2ex}$, $\exists\raisebox{1.1ex}{\scaleto{\sf H}{.8ex}\kern-.2ex}$ of sort ${\sf H}$; and the quantifiers $\forall\raisebox{1.1ex}{\scaleto{\sf S}{.8ex}\kern-.2ex}$, $\exists\raisebox{1.1ex}{\scaleto{\sf S}{.8ex}\kern-.2ex}$ of sort ${\sf S}$.

We write \emph{$x\in C$\/} for the predicate associated to a compact set $C$, and denote by $|\mbox{-}|$ the length of a tuple.

\begin{definition}\label{def_atomic}
  The atomic formulas of ${\EuScript F}$ are
  \begin{itemize}
  \item[i.] atomic and negated atomic formulas of ${\EuScript L}_{\sf H}$
  \item[ii.] those of the form $\tau\in C$, where $\tau$ is a term of sort ${\sf H}^n\times {\sf S}^m\to {\sf S}$, and $C$ a compact set.
  \end{itemize}
\end{definition}

We remark that a smaller fragment of ${\EuScript F}$ may be of interest in some specific contexts.
This fragment is obtained by excluding equalities and inequalities from (i).
The discussion below applies to this smaller fragment as well.

Let $M\subseteq N$ be standard structures.
We say that $M$ is an \emph{${\EuScript F}$-elementary\/} substructure of $N$ if the latter models all ${\EuScript F}(M)$-sentences that are true in $M$.

Let $x$ and $\xi$ be variables of sort ${\sf H}$, respectively ${\sf S}$.
A standard structure $M$ is \emph{${\EuScript F}$-saturated\/} if it realizes every type $p(x\,;\xi)\subseteq{\EuScript F}(A)$, for any $A\subseteq M$ of cardinality smaller than $|M|$, that is finitely consistent in $M$.
The following theorem is proved in \cite{clcl} for signatures that do not contain symbols of sort ${\sf H}^n\times{\sf S}^m\to {\sf S}$ with $n{\cdot}m>0$.
A similar framework has been independently introduced in \cite{CP}~--~only without quantifiers of sort {\sf H} and with an approximate notion of satisfaction.
The corresponding compactness theorem is proved there with a similar method.
The setting in \cite{CP} is a generalization of that used by Henson and Iovino for Banach spaces~\cite{HI}.
By the elimination of quantifiers of sort {\sf S}, proved in \cite{clcl}*{Proposition 3.6}, the two approaches are equivalent~--~up to approximations.

In \cite{Z} it is observed that, under the assumption of equicontinuity, the proof in \cite{clcl} extends to the case $n{\cdot}m> 0$.

\begin{theorem}[ (Compactness)]\label{thm_compactness}
  Every standard structure has an ${\EuScript F}$-saturated ${\EuScript F}$-elementary extension.
\end{theorem}

We introduce a new sort ${\sf X}$ with the purpose of conveniently describing classes of ${\EuScript F}$-for\-mulas that only differ by the relation symbol $C\subseteq S$ they contain.
Let ${\EuScript F}_{\sf X}$ be defined as ${\EuScript F}$ but replacing (ii) in Definition~\ref{def_atomic} by\smallskip

\begin{itemize}
  \item[iii.] $\tau(x\,;\xi)\in X$, where $X$ is a variable of sort ${\sf X}$.
\end{itemize}

Formulas in ${\EuScript F}_{\sf X}$ are denoted by $\varphi(x\,;\xi;X)$, where $X=X_1,\dots,X_n$ is a tuple of variables of sort ${\sf X}$.
If $C=C _1,\dots,C_n$ is a tuple of compact subsets, then $\varphi(x\,;\xi;C)$, is a formula in ${\EuScript F}$.
This we call an \emph{instance\/} of $\varphi(x\,;\xi;X)$.
All formulas in ${\EuScript F}$ are instances of formulas in ${\EuScript F}_{\sf X}$.

\begin{definition}
  Let $\varphi\in{\EuScript F}_{\sf X}$ be a formula~--~possibly with some hidden free variables.
  The \emph{pseudonegation\/} of $\varphi\in{\EuScript F}_{\sf X}$ is the formula obtained by replacing the atomic formulas in  ${\EuScript L}_{\sf H}$ by their negation and the logical symbols $\wedge$, $\vee$, $\forall$, $\exists$
  by their respective dual $\vee$, $\wedge$, $\exists$, $\forall$.
  The atomic formulas of the form $\tau\in X$ remain unchanged.\smallskip
  
The pseudonegation of $\varphi$ is denoted by \emph{${\sim}\varphi$.}
Clearly, when no variable of sort ${\sf X}$ occurs in $\varphi$, we have ${\sim}\varphi\leftrightarrow\neg\varphi$.
\end{definition}

Note that, unlike $\neg\varphi$, the pseudonegation ${\sim}\varphi$ is again a formula in ${\EuScript F}_{\sf X}$.

It is sometimes convenient to work with infinite conjunctions of formulas.
In the following $\sigma(x\,;z\,;X)$ denotes an ${\EuScript F}_{\sf X}$-type.
We write ${\sim}\sigma(x\,;z\,;X)$ for the possibly infinite disjunction of the formulas ${\sim}\varphi(x\,;z\,;X)$ for $\varphi(x\,;z\,;X)\in\sigma$.

When $C$ and $\tilde C$ are tuples we read $C\subseteq C'$ and $\tilde C\cap C=\varnothing$ componentwise.
We say that $C'$ is a neighborhood of $C$ if every component of $C'$ contains an open set containing the corresponding component of $C$.

\begin{fact}\label{fact_trivial}
  The following hold for every $C\subseteq C'$ and $\tilde C\cap C=\varnothing$, every standard structure $M$, and every ${\EuScript F}_{\sf X}(M)$-type $\sigma(X)$\smallskip

  \ceq{\hfill M}{\models}{\phantom{\sim}\sigma(C)\ \rightarrow\phantom{\neg}\sigma(C')}\smallskip

  \ceq{\hfill M}{\models}{{\sim}\sigma(\tilde C)\ \rightarrow\neg\sigma(C).}\smallskip
  
\end{fact}

\begin{fact}
  For every neighborhood $C'$ of $C$ there is $\tilde C$ disjoint from $C$ such that\smallskip
  
  \ceq{\hfill M}{\models}{\varphi(C)\ \rightarrow\ \neg\,{\sim}\varphi(\tilde C)\ \rightarrow\ \varphi(C')}\smallskip

  for every $M$, and every ${\EuScript F}_{\sf X}(M)$-formula $\varphi(X)$.
  Conversely, for every $\tilde C\cap C=\varnothing$ there is a neighboorhood $C'$ of $C$ such that

  \ceq{\hfill M}{\models}{\varphi(C)\ \rightarrow\ \varphi(C')\ \rightarrow\ \neg\,{\sim}\varphi(\tilde C)}\smallskip

  for every $M$, and every ${\EuScript F}_{\sf X}(M)$-formula $\varphi(X)$.
\end{fact}

\begin{proof}
  When $\varphi(X)$ is the atomic formula $\tau\in X$, the first claim holds with $\tilde C=S\smallsetminus O$ where $O$ is any open set $C\subseteq O\subseteq C'$.
  The second claim holds when $C'$ is any neighborhood of $C$ disjoint from $\tilde C$.
  Induction on the syntax of $\varphi(X)$ proves the general case.
\end{proof}

The above fact has the following useful consequence.

\begin{fact}\label{fact_otto}
  The following are equivalent for every standard structure $M$ , every ${\EuScript F}_{\sf X}(M)$-type $\sigma(X)$, and every $C$\smallskip
  
    \ceq{1.\hfill M}{\models}{\phantom{\neg}\sigma(C)\ \leftarrow\ \bigwedge\Big\{\ \sigma(C')\ :\ C'\textrm{ neighborhood of }C\Big\} }\smallskip

    \ceq{2.\hfill M}{\models}{\neg\sigma(C)\ \rightarrow\ \bigvee\Big\{{\sim}\sigma(\tilde C)\ :\ \tilde C\cap C=\varnothing\Big\}.}\smallskip

  (The converse implications are trivial~--~therefore not displayed.)
\end{fact}

\begin{fact}\label{fact_saturation}
  The equivalent conditions in Fact~\ref{fact_otto} hold in all ${\EuScript F}$-saturated structures.
\end{fact}

\begin{proof}
  \def\medrel#1{\parbox{5ex}{\hfil $#1$}}
  \def\ceq#1#2#3{\parbox[t]{39ex}{$\displaystyle #1$}\medrel{#2}{$\displaystyle #3$}}

  It suffices to prove the claim for formulas.
  We prove the first of the two implications by induction on the syntax of $\varphi(X)$.
  The existential quantifiers of sort {\sf H} and {\sf S} are the only cases that require attention.
  Assume inductively that

    \ceq{\hfill\bigwedge\Big\{\ \varphi(a\,;C')\ :\ C'\textrm{ neighborhood of }C\Big\}}{\rightarrow}{\varphi(a\,;C)}

  holds for every $\varphi(a\,;C)$.
  Then induction for the existential quantifier follows from

    \ceq{\hfill\bigwedge\Big\{\ \exists x\,\varphi(x\,;C')\ :\ C'\textrm{ neighborhood of }C\Big\}}{\rightarrow}{\exists x\,\bigwedge\Big\{\varphi(x\,;C')\ :\ C'\textrm{ neighborhood of }C\Big\} }

    which is a consequence of saturation and the fact that if $C'$, $C''$ are compact neighbourhoods of $C$ then $\exists x\, \varphi(x;C'\cap C'')$ is satisfiable.
    The proof for the existential quantifier of sort {\sf S} is similar.
\end{proof}

We say that $M$ is \emph{${\EuScript F}$-maximal\/} if it models $\varphi\in{\EuScript F}(M)$ whenever $\varphi$ holds in some ${\EuScript F}$-elemen\-tary extensions of $M$.
By Fact~\ref{fact_saturation} and the following fact, all ${\EuScript F}$-saturated standard structures are ${\EuScript F}$-maximal.

\begin{fact}\label{fact_maximal}
  Let $M$ be a standard structure.
  Then the following are equivalent
  \begin{itemize}
    \item [1.] $M$ is ${\EuScript F}$-maximal
    \item [2.] the conditions in Fact~\ref{fact_otto} hold for every type (equivalently, for every formula). 
  \end{itemize}
\end{fact}

\begin{proof}
  1$\Rightarrow$2.
  If (2) in Fact~\ref{fact_otto} fails, $\neg\varphi(C)\wedge\neg\,{\sim}\varphi(\tilde C)$ holds in $M$ for every $\tilde C\cap C=\varnothing$.
  Let ${\EuScript U}$ be an ${\EuScript F}$-saturated ${\EuScript F}$-elementary extension of $M$.
  As $M$ is ${\EuScript F}$-maximal, $\neg\varphi(C)\wedge\neg\,{\sim}\varphi(\tilde C)$ also holds in ${\EuScript U}$.
  Then (2) in Fact~\ref{fact_otto} fails in ${\EuScript U}$, contradicting Fact~\ref{fact_saturation}.

  2$\Rightarrow$1. 
  Let $N$ be any ${\EuScript F}$-elementary extension of $M$.
  Suppose $N\models\varphi(C)$.
  By (2), it suffices to prove that $M\models\varphi(C')$ for every neighborhood $C'$ of $C$.
  Suppose not.
  Then, by (2) in Fact~\ref{fact_otto}, $M\models{\sim}\varphi(\tilde C)$ for some $\tilde C\cap C'=\varnothing$.
  Then $N\models{\sim}\varphi(\tilde C)$ by ${\EuScript F}$-elementarity.
  As $\tilde C\cap C=\varnothing$, this contradicts Fact~\ref{fact_trivial}.
\end{proof}

\begin{notation}
  In what follows we work in a fixed ${\EuScript F}$-saturated standard structure ${\EuScript U}$.
  We denote by $\kappa$ the cardinality of ${\EuScript U}$.
  We assume that $\kappa$ is a Ramsey cardinal larger than the cardinality of the language.
  However, the willing reader may check that we can do without large cardinals by careful application of the Erd\H{o}s-Rado Theorem, see \cite{TZ}*{Appendix C.3}.\smallskip

  In the following $\sigma(x\,;z\,;X)$ always denotes an ${\EuScript F}_{\sf X}({\EuScript U})$-type of small (i.e.\@ $<\kappa$) cardinality.
\end{notation}

\section{A duality in \boldmath$K(S)$}\label{K(S)}

\def\medrel#1{\parbox{5ex}{\hfil $#1$}}
\def\ceq#1#2#3{\parbox[t]{23ex}{$\displaystyle #1$}\medrel{#2}{$\displaystyle #3$}}

In this section we introduce some notions that are used to describe global types in Section~\ref{nonfinitary}.

Write \emph{$K(S)$\/} for the set of compact subsets of $S$.
In this section ${\EuScript D}$ and ${\EuScript C}$ range over subsets of $K(S)$.
We define 

\ceq{\hfill\emph{${\sim}{\EuScript D}$}}{=}{\big\{\tilde C\ :\ \tilde C\cap C\neq\varnothing\text{ for every }C\in{\EuScript D}\big\}.}

Above we read $\tilde C\cap C\neq\varnothing$ as the negation of $\tilde C\cap C=\varnothing$.
That is, some components of the tuples $\tilde C$ and $C$ have nonempty intersection.

The following fact motivates the definition.

\begin{fact}\label{fact_~definibile}
  Let ${\EuScript D}=\{C\,:\,\varphi(C)\}$.
  Then ${\sim}{\EuScript D}=\{\tilde C\,:\,{\sim}\varphi(\tilde C)\}$.
\end{fact}

\begin{proof}
  We need to prove that 

  \ceq{\hfill{\sim}\varphi(\tilde C)}{\Leftrightarrow}{\text{for every } C,\ \text{ if } C\cap\tilde C=\varnothing\ \text{ then } \neg\varphi(C).}

  The implication $\Rightarrow$ follows immediately from Fact~\ref{fact_trivial}.
  To prove $\Leftarrow$ assume the r.h.s.
  Then $\neg\varphi(S\smallsetminus O)$ holds for every open set  $O\supseteq\tilde C$.
  From the second implication in Fact~\ref{fact_otto} we obtain ${\sim}\varphi(\tilde C')$ for some $\tilde C\subseteq\tilde C'\subseteq O$.
As $O$ is arbitrary, we obtain ${\sim}\varphi(\tilde C)$ from the first implication in Fact~\ref{fact_otto}.
\end{proof}

We prove some straightforward inclusions.

\begin{fact}\label{fact_~inclusione}
  For every ${\EuScript D}\subseteq{\EuScript C}$ we have ${\sim}{\EuScript C}\subseteq{\sim}{\EuScript D}$.  
  Moreover, ${\EuScript D}\subseteq{\sim}{\sim}{\EuScript D}$.
\end{fact}

\begin{proof}
  Let $C\notin{\sim}{\EuScript D}$.
  Then $\tilde C\cap C=\varnothing$ for some $\tilde C\in{\EuScript D}$.
  As  $\tilde C\in{\EuScript C}$ we conclude that  $C\notin{\sim}{\EuScript C}$.
  For the second claim, let $C\notin{\sim}{\sim}{\EuScript D}$.
  Then $\tilde C\cap C=\varnothing$ for some $\tilde C\in{\sim}{\EuScript D}$.
  Then $C\notin{\EuScript D}$ follows.
\end{proof}

\begin{fact}
  For every ${\EuScript D}$ the following are equivalent
  \begin{itemize}
    \item [1.] ${\EuScript D}={\sim}{\sim}{\EuScript D}$.
    \item [2.] ${\EuScript D}={\sim}{\EuScript C}$ for some ${\EuScript C}$.
  \end{itemize}
\end{fact}
  
\begin{proof}
  Only 2$\Rightarrow$1 requires a proof.
  From Fact~\ref{fact_~inclusione} we obtain ${\sim}{\sim}{\sim}{\EuScript C}\subseteq{\sim}{\EuScript C}$.
  Assume (2).
  Then ${\sim}{\sim}{\EuScript D}\subseteq{\EuScript D}$ which, again by Fact~\ref{fact_~inclusione}, suffices to prove (1).
\end{proof}

We say that ${\EuScript D}$ is \emph{involutive\/} if ${\EuScript D}={\sim}{\sim}{\EuScript D}$.

\begin{theorem}
  The following are equivalent
  \begin{itemize}
    \item [1.] ${\EuScript D}$ is involutive 
    \item [2.] \noindent\kern-\labelwidth\kern-\labelsep
    \ceq{\hfill C\in{\EuScript D}}{\Leftrightarrow}{C'\in{\EuScript D}} for every neighborhood $C'$ of $C$.
  \end{itemize}
\end{theorem}

It is not difficult to see that (2) holds if and only if ${\EuScript D}$ is closed in the product of Vietoris topologies and includes all the supersets of its elements.

\begin{proof}
  2$\Rightarrow$1.
  By Fact~\ref{fact_~inclusione} it suffices to prove ${\sim}{\sim}{\EuScript D}\subseteq{\EuScript D}$.
  Let $C\in{\sim}{\sim}{\EuScript D}$.
  Let $O$ be a tuple of open sets containing $C$ componentwise.
  Let $\tilde C$ be the componentwise complement of $O$.  
  Then $\tilde C\notin{\sim}{\EuScript D}$.
  As ${\sim}{\EuScript D}=\{\tilde C:  \hat C\cap\tilde C\neq\varnothing\text{ for every }\hat C\in{\EuScript D}\}$, we have that $\hat C\in{\EuScript D}$ for some $\hat C\subseteq O$.
  Assume (2).
  Then, by $\Rightarrow$ every $C'\supseteq O$ is in ${\EuScript D}$.
  As $O$ is arbitrary, $C\in {\EuScript D}$ by $\Leftarrow$.

  1$\Rightarrow$2.
  It suffices to show that (2) holds with ${\sim}{\EuScript D}$ for ${\EuScript D}$.
  The implication $\Rightarrow$ in (2) is obvious.
  To prove $\Leftarrow$, assume $C\notin{\sim}{\EuScript D}$.
  Then $C\cap\tilde C=\varnothing$ for some $\tilde C\in{\EuScript D}$.
  Let $C'$ be a neighboorhood of $C$ disjoint from $\tilde C$.
  Then the r.h.s.\@ of the equivalence fails for $C'$.  
\end{proof}

\section{Stable formulas~--~the finitary case}\label{finitary}
\def\medrel#1{\parbox{5ex}{\hfil $#1$}}
\def\ceq#1#2#3{\parbox[t]{22ex}{$\displaystyle #1$}\medrel{#2}{$\displaystyle #3$}}

In this section we consider the strongest possible notion of order property.
It is not the most interesting one, but it is very similar to the classical definition~--~the comparison may be useful, though not stricly required for the sequel.

The following definition of stability is the classical one.
We dub it \textit{finitary\/} for the reason explained below.

\begin{definition}\label{def_finitary_stable}\strut
  Let $\varphi(x\,;z\,;X)$, for $x$ and $z$ tuples of variables of sort ${\sf H}$, be an ${\EuScript F}_{\sf X}$-for\-mu\-la.  
  Let $C$ be a fixed tuple of compact sets.
  We say that $\varphi(x\,;z\,;C)$ is \emph{finitarily unstable\/} if for every $m<\omega$ there is a sequence $\langle a_i\,;b_i\ :\ i<m\rangle$ such that\smallskip

  \ceq{1.\hfill \varphi(a_n\,;b_i\,;C)}{\wedge}{\neg\varphi(a_i\,;b_n\,;C)}\hfill for every $i<n<m$.\smallskip

  A formula is \emph{finitary stable\/} if it is not finitarily unstable.
\end{definition}

The following is an equivalent definition which only mentions formulas in ${\EuScript F}$.
The reader may wish to compare it with Definition~\ref{def_stable}.

\begin{fact}\label{fact_stability_semicalssic}
  The following are equivalent
  \begin{itemize}
    \item [1.]  $\varphi(x\,;z\,;C)$ is finitarily stable
    \item [2.]  for some $m$ there is no sequence $\langle a_i\,;b_i\ :\ i<m\rangle$ such that for some $\tilde C\cap C=\varnothing$\smallskip
    
    \noindent\kern-\labelwidth\kern-\labelsep
    \ceq{\hfill \varphi(a_n\,;b_i\,;C)}{\wedge}{{\sim}\varphi(a_i\,;b_n\,;\tilde C)}\hfill for every $i<n<m$.
  \end{itemize}
\end{fact}

\begin{proof}
  1$\Rightarrow$2.
  By Fact~\ref{fact_trivial}.

  2$\Rightarrow$1.
  Negate (1).
  Let $\langle a_i\,;b_i\ :\ i<m\rangle$ be such that (1) of Definition~\ref{def_finitary_stable} holds.
  By Fact~\ref{fact_otto} there are $\tilde C_{n,i}\cap C=\varnothing$ such that 

  \ceq{\hfill \varphi(a_n\,;b_i\,;C)}{\wedge}{{\sim}\varphi(a_i\,;b_n\,;\tilde C_{n,i})}\hfill for every $i<n<m$.

  Then the union of the $\tilde C_{n,i}$ is the set $\tilde C$ required to negate (2).
\end{proof}

In a classical context, i.e.\@ relying on ${\EuScript L}$-saturation, one can replace $m$ in Definition~\ref{def_finitary_stable} by $\omega$.
But this may not be true if we only have ${\EuScript F}$-saturation.
Therefore we dedicate a separate section to the non-finitary version of stability.

The following definitions are classical, i.e.\@ they rely on the full language ${\EuScript L}$.

\begin{definition}\label{def_globaltype}\strut
  For a given $C$, a \emph{global $\varphi(x\,;z\,;C)$-type\/} is a maximally (finitely) consistent set of formulas of the form $\varphi(x\,;b\,;C)$ or $\neg\varphi(x\,;b\,;C)$ for some $b\in{\EuScript U}^{z}$.
\end{definition}

Global $\varphi(x\,;z\,;C)$-types should not be confused with global $\varphi(x\,;z\,;X)$-types which are introduced in the following section.

In this section the symbol ${\EuScript D}$ always denotes a subset of ${\EuScript U}^z$.

\begin{definition}\label{def_approx}\strut
  We say that ${\EuScript D}$ is \emph{approximable\/} by $\varphi(x\,;z\,;C)$ if for every finite $B\subseteq {\EuScript U}^z$ there is an $a\in {\EuScript U}^x$ such that\smallskip

  \ceq{\hfill b\in {\EuScript D}}{\Leftrightarrow}{\varphi(a\,;b\,;C)}\hfill for every $b\in B$.
\end{definition}

It goes without saying that these definitions describe two sides of the same coin.
In fact, it is clear that ${\EuScript D}$ is approximable by $\varphi(x\,;z\,;C)$ if and only if  
    
  \noindent\kern-\labelwidth\kern-\labelsep
  \ceq{\hfill p(x)}{=}{\big\{\varphi(x\,;b\,;C)\ :\ b\in{\EuScript D}\big\}\ \ \cup\ \ \big\{\neg\varphi(x\,;b\,;C)\ :\ b\notin{\EuScript D}\big\}}

is a global $\varphi(x\,;z\,;C)$-type.\smallskip

The following theorem is often rephrased by saying that, when $\varphi(x\,;z\,;C)$ is finitarily stable, all global $\varphi(x\,;z\,;C)$-types are definable.
The proof of the classical case applies verbatim because no saturation is required.
We only state it here for comparison with its non-finitary counterpart, Theorem~\ref{thm_stability_definability}, and its approximate counterpart, Theorem~\ref{thm_epsilon_stability_definability}.

\begin{theorem}
  Let $\varphi(x\,;z\,;C)$ be finitarily stable.
  Let ${\EuScript D}$ be approximable by $\varphi(x\,;z\,;C)$.
  Then there are some $\langle a_{i,j}\ :\ i< k,\ j<m\rangle$ such that for every $b\in{\EuScript U}^z$\medskip

  \ceq{\hfill b\in{\EuScript D}}{\Leftrightarrow}{\bigvee_{i< k}\ \bigwedge_{j<m}\ \varphi(a_{i,j}\,;b\,;C).}
\end{theorem}

\section{Stable formulas~--~the non-finitary case}\label{nonfinitary}
\def\medrel#1{\parbox{5ex}{\hfil $#1$}}
\def\ceq#1#2#3{\parbox[t]{22ex}{$\displaystyle #1$}\medrel{#2}{$\displaystyle #3$}}

We discuss a version of stability that seems more appropriate in our context.
There are a few differences with the previous section.
First, here the attribute \textit{stable\/} refers to a whole class of formulas~--~all instances of some $\varphi(x\,;z\,;X)$.
Furthermore, the sequence witnessing the order property is required to be infinite.
Finally, since we need to extend the definition of stability to types, sequences of length $\kappa$ (the cardinality of ${\EuScript U}$, a Ramsey cardinal) are used.

In this section  ${\EuScript D}$ is always a subset of ${\EuScript U}^z\times K(S)^{|X|}$.
We write ${\sim}{\EuScript D}$ for the set obtained by applying the definition in Section~\ref{K(S)} to all fibers of ${\EuScript D}$.
That is, if ${\EuScript D}_b$ is the $b$-fiber $\big\{C\,:\,\langle b\,;C\rangle\in{\EuScript D}\big\}$, we define

\ceq{\hfill{\sim}{\EuScript D}}{=}{\bigcup_{b\in{\EuScript U}^z}\{b\}\times({\sim}{\EuScript D}_b).}

We say that ${\EuScript D}$ is involutive if all its fibers are.

\begin{definition}\label{def_stable}\strut
  Let $x$ and $z$ be tuples of variables of sort ${\sf H}$.
  Recall that $\sigma(x\,;z\,;X)$ always denotes a small ${\EuScript F}_{\sf X}({\EuScript U})$-type.\smallskip

  We say that $\sigma(x\,;z\,;X)$ is \emph{unstable\/} if there are a sequence $\langle a_i\,;b_i\ :\ i<\kappa\rangle$, some $C$, and some $\tilde C\cap C=\varnothing$ such that\smallskip

  \ceq{\hfill \sigma(a_n\,;b_i\,;C)}{\wedge}{{\sim}\sigma(a_i\,;b_n\,;\tilde C)}\hfill for every $i<n<\kappa$ or for every $n<i<\kappa$.\smallskip

  A type is \emph{stable\/} if it is not unstable.
\end{definition}

By saturation, when $\sigma(x\,;z\,;X)$ is a formula we can replace $\kappa$ with $\omega$ in the previous definition

\begin{definition}\strut 
  Let $\varphi(x\,;z\,;X)$ be an ${\EuScript F}_{\sf X}$-formula.
  A \emph{global $\varphi(x\,;z\,;X)$-type\/} is a maximally (finitely) consistent set of formulas of the form $\varphi(x\,;b\,;C)$ or ${\sim}\varphi(x\,;b\,;C)$ for some $b\in{\EuScript U}^{z}$ and some $C$.
\end{definition}

It is convenient to have a different characterization of global types, therefore  the following definition.

\begin{definition}\label{def_approx_X}\strut
  We say that ${\EuScript D}$ is \emph{approximable\/} by $\sigma(x\,;z\,;X)$ if for every small set $B\subseteq{\EuScript U}^z$ there is an $a\in{\EuScript U}^x$ such that\smallskip

  \ceq{1.\hfill \langle b,C\rangle\in{\EuScript D}}{\Rightarrow}{\sigma(a\,;b\,;C)}\quad and\smallskip

  \ceq{2.\hfill\langle b,C\rangle\in {\EuScript D}}{\Leftarrow}{\sigma(a\,;b\,;C)}\hfill for every $b\in B$ and every $C$.\smallskip
\end{definition}

When $\sigma(x\,;z\,;X)$ is a formula and  ${\EuScript D}$ is involutive (2) in the above definition can be rephrased in a more convenient form.

\begin{remark}\label{rem_approx}
  If ${\EuScript D}$ is involutive, the following are equivalent\smallskip

  \ceq{2.\hfill\langle b,C\rangle\in {\EuScript D}}{\Leftarrow}{\varphi(a\,;b\,;C)}\smallskip

  \ceq{2'\hfill\langle b,C\rangle\in{\sim}{\EuScript D}}{\Rightarrow}{{\sim}\varphi(a\,;b\,;C)}.
\end{remark}

\begin{proof}
  Let ${\EuScript C}=\{\langle b,C\rangle\,:\,\varphi(a,b,C)\}$.
  This is also an involutive set by Fact~\ref{fact_~definibile}.
  Then (2) says ${\EuScript C}\subseteq{\EuScript D}$.
  By Fact~\ref{fact_~inclusione} this is equivalent to ${\sim}{\EuScript D}\subseteq{\sim}{\EuScript C}$ which in turn is equivalent to (2$'$), again by Fact~\ref{fact_~definibile}.
\end{proof}

Under the assumptions of the remark (i.e. when ${\EuScript D}$ is involutive and $\sigma(x\,;z\,;X)$ is a formula), it is easy to verify by saturation that in Definition~\ref{def_approx_X} the requirement \textit{for every small\/} $B$ can be replaced by \textit{for every finite\/} $B$ (indeed, one can verify that the proof of the following fact also works when $B$ is required to be finite).

\begin{fact}
  Let $\varphi(x\,;z\,;X)$ be an ${\EuScript F}_{\sf X}$-formula.
  For every involutive ${\EuScript D}$ the following are equivalent
  \begin{itemize}
    \item [1.] ${\EuScript D}$ is approximable by $\varphi(x\,;z\,;X)$
    \item [2.] the following is a global $\varphi(x\,;z\,;X)$-type\smallskip
    
    \noindent\kern-\labelwidth\kern-\labelsep
    \ceq{\hfill p(x)}{=}{\big\{\phantom{\sim}\varphi(x\,;b\,;C)\ :\ \langle b,C\rangle\in{\EuScript D}\big\}\ \ \cup}\smallskip

    \noindent\kern-\labelwidth\kern-\labelsep
    \ceq{}{~}{\big\{{\sim}\varphi(x\,;b\,;C)\ :\ \langle b,C\rangle\in{\sim}{\EuScript D}\big\}.}
  
  \end{itemize}
\end{fact}

\begin{proof}
  2$\Rightarrow$1.
  Let $B$ be a given small subset of ${\EuScript U}^z$.
  The consistency of $p(x)\restriction B$ provides some $a$ such that

  \ceq{\hfill \langle b,C\rangle\in {\EuScript D}}{\Rightarrow}{\phantom{\sim}\varphi(a\,;b\,;C)}\quad and 
  
  \ceq{\hfill \langle b,C\rangle\in{\sim}{\EuScript D}}{\Rightarrow}{{\sim}\varphi(a\,;b\,;C)}\hfill for every $b\in B$ and every $C$. 

  By Remark~\ref{rem_approx}, these yield (1) and (2) in Definition~\ref{def_approx_X}.
  
  1$\Rightarrow$2.
  First we prove that any $p(x)$ defined as in (2) is maximal whenever it is consistent~--~in fact, this only depends on ${\EuScript D}$ being involutive.
  We need to consider two cases.
  First, assume that $\varphi(x\,;b\,;C)$ is consistent with $p(x)$.
  Then consistency implies that $\langle b,\tilde C\rangle\notin{\sim}{\EuScript D}$ for every $\tilde C\cap C=\varnothing$.
  But this implies that $\langle b, C\rangle\in{\sim}{\sim}{\EuScript D}={\EuScript D}$, therefore $\varphi(x\,;b\,;C)\in p$.
  The second case is when ${\sim}\varphi(x\,;b\,;C)$ is consistent with $p(x)$.
  This implies that $\langle b,\tilde C\rangle\notin{\EuScript D}$ for every $\tilde C\cap C=\varnothing$.
  Then $\langle b,C\rangle\in{\sim}{\EuScript D}$, therefore ${\sim}\varphi(x\,;b\,;C)\in p$.

  Finally, to prove consistency, let $B$ be a given small set.
  Let $a$ be as in Definition~\ref{def_approx_X}.
  Then $a\models p(x)\restriction B$ by Remark~\ref{rem_approx}.
\end{proof}

The proof of the main theorem of this section requires the following notion of approximation.
The definition is inspired by the classical notion of approximation from below in~\cite{Z15} where it is used as rephrasing of the notion of \textit{honest definability\/} in~\cite{CS}.
Here this notion plays a purely technical role, but we present it as potentially interesting in its own right.

\begin{definition}\label{def_approx_blw}\strut
  We say that ${\EuScript D}$ is \emph{approximable\/} by $\sigma(x\,;z\,;X)$ \emph{from below\/} if for every small set $B\subseteq{\EuScript U}^z$ there is an $a\in{\EuScript U}^x$ such that for every $C$\smallskip

  \ceq{1.\hfill \langle b,C\rangle\in{\EuScript D}}{\Rightarrow}{\sigma(a\,;b\,;C)}\hfill for every $b\in B$ and\smallskip

  \ceq{2.\hfill\langle b,C\rangle\in {\EuScript D}}{\Leftarrow}{\sigma(a\,;b\,;C)}\hfill for every $b\in{\EuScript U}^z$.\smallskip
\end{definition}

\begin{theorem}\label{thm_stability_definability}
  Let $\sigma(x\,;z\,;X)$ be stable.
  Assume that ${\EuScript D}$ is approximable by $\sigma(x\,;z\,;X)$.
  Then there are some $\langle a_{i,j}\ :\ i,j<\lambda\rangle$ such that for every $b\in{\EuScript U}^z$ and every $C$\medskip

  \ceq{\hfill \langle b,C\rangle\in{\EuScript D}}{\Leftrightarrow}{\bigvee_{i<\lambda}\ \bigwedge_{j<\lambda}\ \sigma(a_{i,j}\,;b\,;C).}\medskip

\end{theorem}

\begin{proof}
  The theorem is an immediate consequence of the following three lemmas.
\end{proof}

\begin{lemma}\label{lem_1_inf}
  Let $\sigma(x\,;z\,;X)$ be stable.
  Assume that ${\EuScript D}$ is approximable by $\sigma(x\,;z\,;X)$ from below.
  Then there are some $\langle a_i\ :\ i<\lambda\rangle$ such that for every $b\in{\EuScript U}^z$ and every $C$\medskip

  \ceq{1.\hfill \langle b,C\rangle\in{\EuScript D}}{\Rightarrow}{\bigvee_{i<\lambda}\ \sigma(a_i\,;b\,;C)}\medskip 

  \ceq{2.\hfill \langle b,C\rangle\in{\EuScript D}}{\Leftarrow}{\bigvee_{i<\lambda}\ \sigma(a_i\,;b\,;C).} 
\end{lemma}

\begin{proof}
  We define recursively the required parameters $a_i$ together with some auxiliary parameters $b_i$.
  The element $a_n$ is chosen so that

\ceq{3.\hfill \langle b,C\rangle\in{\EuScript D}}{\Leftarrow}{\sigma(a_n\,;b\,;C)}\hfill for every $b\in{\EuScript U}^z$ and every $C$

and

\ceq{4.\hfill \langle b_i,C\rangle\in{\EuScript D}}{\Rightarrow}{\sigma(a_n\,;b_i\,;C)}\hfill for every $i<n$ and every $C$.\smallskip

This is possible because ${\EuScript D}$ is approximated from below.
Note that (3) immediately guarantees (2).
Now, assume (1) fails for $\lambda=n+1$, and choose $b_n$ and $C_n$ witnessing this.
Then, by Fact~\ref{fact_otto}, for some $\tilde C_{i,n}\cap C_n=\varnothing$

\ceq{5.\hfill \langle b_n,C_n\rangle\in{\EuScript D}\ \ }{\text{and}}{\ \ \bigwedge_{i\le n}\ {\sim}\sigma(a_i\,;b_n\,;\tilde C_{i,n}).}

Suppose for a contradiction that the construction carries on for $\kappa$ stages.
Then, as (5) guarantees that $\langle b_i,C_i\rangle\in{\EuScript D}$ for every $i$, from (4) we obtain $\sigma(a_n\,;b_i\,;C_i)$ for every $i<n$.
From (5) we also obtain ${\sim}\sigma(a_i\,;b_n\,;\tilde C_{i,n})$, for every $i\le n$.
Apply Ramsey with pairs $\langle C,\tilde C\rangle$ of compact subsets of $S$ as colors to obtain a subsequence $\langle a'_i\,;b'_i\ :\ i<\kappa\rangle$ that contradicts Definition~\ref{def_stable}.
\end{proof}

\begin{lemma}
  Let $\sigma(x\,;z\,;X)$ be stable.
  Assume that ${\EuScript D}$ is approximable by $\sigma(x\,;z\,;X)$.
  Then for some $\lambda<\kappa$ the type\medskip 

  \ceq{\hfill\pi(\bar x\,;z\,;X)}{=}{\bigwedge_{i<\lambda}\ \sigma(x_i\,;z\,;X),}\smallskip
  
  where the $x_i$ are copies of $x$ and $\bar x=\langle x_i\ :\ i<\lambda\rangle$,
  approximates ${\EuScript D}$ from below.
\end{lemma}

\begin{proof}
  Negate the claim and let $B$ witness that $\pi(\bar x)$ does not approximate ${\EuScript D}$ from below.
  Suppose that $\langle a_i:i<n\rangle$ and $\langle b_i:i<n\rangle$ have been defined.
  Choose $a_n$ such that for every $b\in B\cup\{b_i:i<n\}$ and every $C$

  \ceq{1.\hfill\langle b,C\rangle\in{\EuScript D}}{\Rightarrow}{\sigma(a_n\,;b\,;C)}\quad and

  \ceq{2.\hfill\langle b,C\rangle\in{\EuScript D}}{\Leftarrow}{\sigma(a_n\,;b\,;C).}
  
  Note that the latter implication is equivalent to: for every $C$ there is some  $\tilde C\cap C=\varnothing$ such that 
  
  \ceq{3.\hfill \langle b,C\rangle\notin{\EuScript D}}{\Rightarrow}{{\sim}\sigma(a_n\,;b\,;\tilde C)}.

  Now, as the lemma is assumed to fail, we can choose $b_n$ and $C_n$ such that

  \ceq{4.\hfill \langle b_n,C_n\rangle\notin{\EuScript D}\ \ }{\text{and}}{\ \ \bigwedge_{i=0}^n\ \sigma(a_i\,;b_n\,;C_n).}

  Note that (4) ensures that $\langle b_i,C_i\rangle\notin{\EuScript D}$ for every $i$.
  Then there is some $\tilde C_{i,n}$ that witnesses (3) for $\langle b_i,C_i\rangle\notin{\EuScript D}$.
  We claim that the procedure has to stop after $<\kappa$ steps.
  In fact, from (3) we obtain ${\sim}\sigma(a_n\,;b_i\,;\tilde C_{i,n})$ for every $i<n$.
  On the other hand, by (4) we have that $\sigma(a_i\,;b_n\,;C_n)$ for every $i<n$.
  Again, apply Ramsey with pairs $\langle C,\tilde C\rangle$ of compact subsets of $S$ as colors to obtain a subsequence $\langle a'_i\,;b'_i\ :\ i<\kappa\rangle$ that contradicts Definition~\ref{def_stable}.
\end{proof}

\begin{lemma}\label{lem_sigma_stable}
  If $\sigma(x\,;z\,;X)$ is stable, then $\pi(\bar x\,;z\,;X)$ in the previous lemma is stable.
\end{lemma}

\begin{proof}
  Suppose $\pi(\bar x\,;z\,;X)$ is unstable.
  Let $\langle \bar a_i\,;b_i\ :\ i<\kappa\rangle$ be a sequence witnessing instability.
  Suppose ${\sim}\pi(\bar a_i\,;b_n\,;\tilde C)$ holds for every $i<n<\kappa$ (the case $n<i<\kappa$ is identical).
  For every pair $i<n$ we can pick some $j<\lambda$ such that ${\sim}\sigma(a_{i,j}\,;b_n\,;\tilde C)$.
  Interpret $j$ as a color. 
  As $\kappa$ is a Ramsey cardinal, every $\lambda$-coloring of a complete graph of size $\kappa$ has a monochromatic subgraph of size $\kappa$.
  That is, there is some fixed $j<\lambda$ such that ${\sim}\sigma(a_{i,j}\,;b_n\,;\tilde C)$ holds for every pair $i<n$ restricted to range over a subsequence of length $\kappa$.
  Such a subsequence contradicts the stability of $\sigma(x\,;z\,;X)$.
\end{proof}

\section{Yet another version of stability}\label{epsilonstable}
\def\medrel#1{\parbox{5ex}{\hfil $#1$}}
\def\ceq#1#2#3{\parbox[t]{28ex}{$\displaystyle #1$}\medrel{#2}{$\displaystyle #3$}}

In this section we present the notion of $\varepsilon$-stability and prove a version of Theorem~\ref{thm_stability_definability} which, albeit approximate,  only requires finite disjunctions and conjunctions.

In this section, the letter $\varepsilon$ always denotes a compact symmetric neighborhood of the diagonal of $S^2$.
If $C$ is a subset of $S$, we write $C^\varepsilon$ for the set of points that are $\varepsilon$-close to points in $C$.
That is

\ceq{\hfill C^\varepsilon}{=}{\big\{\alpha\in S\ :\ \langle\xi,\alpha\rangle\in \varepsilon\ \text{ for some }\ \xi\in C\big\}.}

When $C$ is a tuple, the definition above applies componentwise.

\begin{definition}\label{def_epsilon_stable}\strut
  Let $\varphi(x\,;z\,;X)$, for $x$ and $z$ tuples of variables of sort ${\sf H}$, be an ${\EuScript F}_{\sf X}$-for\-mu\-la.
  We say that $\varphi(x\,;z\,;X)$ is \emph{$\varepsilon$-unstable\/} if for every $m<\omega$ there is a sequence $\langle a_i\,;b_i\,;C_i\,;\tilde C_i\ :\ i<m\rangle$ such that for every $i<n<m$\smallskip

    \ceq{\hfill \varphi(a_n\,;b_i\,;C_i)}{\wedge}{{\sim}\varphi(a_i\,;b_n\,;\tilde C_n)}\quad and\quad $C_i^\varepsilon\cap\tilde C_i=\varnothing$.\smallskip

  A formula is \emph{$\varepsilon$-stable\/} if it is not $\varepsilon$-unstable.  
\end{definition}

Note that the requirement of $\varepsilon$-stability is stronger the smaller the $\varepsilon$.
It is also easy to see that $\varepsilon$-stable for every positive $\varepsilon$ implies stable.

We state the main theorem of this section, which is proved along the same lines as Theorem~\ref{thm_stability_definability}.

\begin{theorem}\label{thm_epsilon_stability_definability}
  Let $\varphi(x\,;z\,;X)$ be $\varepsilon$-stable.
  Assume that ${\EuScript D}$ is approximable by $\varphi(x\,;z\,;X)$.
  Then there are $k,m<\omega$ and $\langle a_{i,j}\ :\ i< k,\ j<m\rangle$ such that for every $b\in{\EuScript U}^z$ and every $C$\medskip

  \ceq{1.\hfill \langle b,C\rangle\in{\EuScript D}}{\Rightarrow}{\bigvee_{i< k}\ \bigwedge_{j<m}\ \varphi(a_{i,j}\,;b\,;C^\varepsilon)}\medskip

  \ceq{2.\hfill \langle b,C^\varepsilon\rangle\in{\EuScript D}}{\Leftarrow}{\bigvee_{i< k}\ \bigwedge_{j<m}\ \varphi(a_{i,j}\,;b\,;C).}
\end{theorem}

\begin{proof}
  The theorem is an immediate consequence of the following three lemmas.
\end{proof}

We need the following version of Definition~\ref{def_approx_blw}.

\begin{definition}\label{def_e_approx_blw}\strut
  We say that ${\EuScript D}$ is \emph{approximable\/} by $\varphi(x\,;z\,;X)$ \emph{from $\varepsilon$-below\/} if for every finite $B\subseteq{\EuScript U}^z$ there is an $a\in{\EuScript U}^x$ such that for every $C$\smallskip

  \ceq{1.\hfill \langle b,C\rangle\in{\EuScript D}}{\Rightarrow}{\varphi(a\,;b\,;C)}\hfill for every $b\in B$ and\smallskip

  \ceq{2.\hfill\langle b,C^\varepsilon\rangle\in {\EuScript D}}{\Leftarrow}{\varphi(a\,;b\,;C)}\hfill for every $b\in{\EuScript U}^z$.\smallskip
\end{definition}

\begin{lemma}
  Let $\varphi(x\,;z\,;X)$ be $\varepsilon$-stable.
  Assume that ${\EuScript D}$ is approximable by $\varphi(x\,;z\,;X)$ from $\varepsilon$-below.
  Then there are $k<\omega$ and $\langle a_i\ :\ i<k\rangle$ such that for every $b\in{\EuScript U}^z$ and every $C$\medskip

  \ceq{1.\hfill \langle b,C\rangle\in{\EuScript D}}{\Rightarrow}{\bigvee_{i<k}\ \varphi(a_i\,;b\,;C^\varepsilon)}\medskip 

  \ceq{2.\hfill \langle b,C^\varepsilon\rangle\in{\EuScript D}}{\Leftarrow}{\bigvee_{i<k}\ \varphi(a_i\,;b\,;C).} 
\end{lemma}

The proof is entirely similar to that of Lemma~\ref{lem_1_inf}.
It is nevertheless included to show the role of $C^\varepsilon$.

\begin{proof}
  We define recursively the required parameters $a_i$ together with some auxiliary parameters $b_i$.
  The element $a_n$ is chosen so that

\ceq{3.\hfill \langle b,C^\varepsilon\rangle\in{\EuScript D}}{\Leftarrow}{\varphi(a_n\,;b\,;C)}\hfill for every $b\in{\EuScript U}^z$ and every $C$

and

\ceq{4.\hfill \langle b_i,C\rangle\in{\EuScript D}}{\Rightarrow}{\varphi(a_n\,;b_i\,;C)}\hfill for every $i<n$ and every $C$.\smallskip

This is possible because ${\EuScript D}$ is approximated from $\varepsilon$-below.
Note that (3) immediately guarantees (2).
Now, assume (1) fails for $k=n+1$, and choose $b_n$ and $C_n$ witnessing this.
Then, by Fact~\ref{fact_otto}, for some $\tilde C_{i,n}\cap C_n^\varepsilon=\varnothing$

\ceq{5.\hfill \langle b_n,C_n\rangle\in{\EuScript D}\ \ }{\text{and}}{\ \ \bigwedge_{i\le n}\ {\sim}\varphi(a_i\,;b_n\,;\tilde C_{i,n}).}

Let $m$ witness the $\varepsilon$-stability of $\varphi(x\,;z\,;X)$.
Suppose for a contradiction that the construction carries on for $m$ stages.
Then, as (5) guarantees that $\langle b_i,C_i\rangle\in{\EuScript D}$ for every $i$, from (4) we obtain $\varphi(a_n\,;b_i\,;C_i)$ for every $i<n<m$.
From (5) we also obtain, for every $i< n<m$, some $\tilde C_{i,n}\cap C_n^\varepsilon=\varnothing$ such that ${\sim}\varphi(a_i\,;b_n\,;\tilde C_{i,n})$.
By monotonicity, ${\sim}\varphi(a_i\,;b_n\,;\tilde C_n)$ for $\tilde C_n=\bigcup_{i< m}\tilde C_{i,n}$.
This contradicts our choice of $m$.
\end{proof}

\begin{lemma}
  Let $\varphi(x\,;z\,;X)$ be $\varepsilon$-stable.
  Assume that ${\EuScript D}$ is approximable by $\varphi(x\,;z\,;X)$.
  Let $m$ be maximal such that a sequence as in Definition~\ref{def_epsilon_stable} exists for $\varphi(x\,;z\,;X)$.
  Write $\bar x$ for $\langle x_i\ :\ i<m\rangle$, where the $x_i$ are copies of $x$.
  Then the formula\smallskip

  \ceq{\hfill\sigma(\bar x\,;z\,;X)}{=}{\bigwedge_{i<m}\ \varphi(x_i\,;z\,;X)}\smallskip

  approximates ${\EuScript D}$ from $\varepsilon$-below.
\end{lemma}

\begin{proof}
  Negate the claim and let $B$ witness that $\sigma(\bar x)$ does not approximate ${\EuScript D}$ from $\varepsilon$-below.
  For $n\le m$ suppose that $a_0,\dots,a_{n-1}$ and $b_0,\dots,b_{n-1}$ have been defined.
  Choose $a_n$ such that for every $b\in B\cup\{b_0,\dots,b_{n-1}\}$ and every $C$

  \ceq{1.\hfill\langle b,C\rangle\in{\EuScript D}}{\Rightarrow}{\varphi(a_n\,;b\,;C)}\quad and

  \ceq{2.\hfill\langle b,C^\varepsilon\rangle\in{\EuScript D}}{\Leftarrow}{\varphi(a_n\,;b\,;C^\varepsilon).}
  
  Note that the second implication is equivalent to: for every $C$ there is some  $\tilde C\cap C^\varepsilon=\varnothing$ such that 
  
  \ceq{3.\hfill \langle b,C^\varepsilon\rangle\notin{\EuScript D}}{\Rightarrow}{{\sim}\varphi(a_n\,;b\,;\tilde C)}.

  Now, as the lemma is assumed to fail, we can choose $b_n$ and $C_n$ such that

  \ceq{4.\hfill \langle b_n,C_n^\varepsilon\rangle\notin{\EuScript D}\ \ }{\text{and}}{\ \ \bigwedge_{i=0}^n\ \varphi(a_i\,;b_n\,;C_n).}

  Note that (4) ensures that $\langle b_i,C_i^\varepsilon\rangle\notin{\EuScript D}$ for every $i$.
  Then there is some $\tilde C_{i,n}\cap C^\varepsilon_i=\varnothing$ that witnesses (3) for $\langle b_i,C_i^\varepsilon\rangle\notin{\EuScript D}$.
  We claim that the procedure has to stop at some stage $<m$. 
  Assume for a contradiction that we can continue till we have $\langle a_i\,;b_i\,;C_i\,;\tilde C_i\ :\ i\le m\rangle$.
  From (3) we obtain ${\sim}\varphi(a_n\,;b_i\,;\tilde C_{i,n})$ for every $i<n\le m$.
  Therefore, by monotonicity, ${\sim}\varphi(a_n\,;b_i\,;\tilde C_i)$ for $\tilde C_i=\bigcup_{n\le m}\tilde C_{i,n}$.
  On the other hand, by (4) we have that $\varphi(a_i\,;b_n\,;C_n)$ for every $i<n\le m$.
  Therefore $\langle a_{m-i}\,;b_{m-i}\,;C_{m-i}\,;\tilde C_{m-i}\ :\ i\le m\rangle$ contradicts the maximality of $m$.
\end{proof}

\begin{lemma}\label{lem_sigma_stable}
  If $\varphi(x\,;z\,;X)$ is $\varepsilon$-stable then $\sigma(\bar x\,;z\,;X)$ in the previous lemma is $\varepsilon$-stable.
\end{lemma}

\begin{proof}
  Let $m$ be maximal such that a sequence as in Definition~\ref{def_epsilon_stable} exists for $\varphi(x\,;z\,;X)$.
  Let $k$ be sufficiently large so that every $m$-coloring of a graph of size $k$ has a monocromatic subgraph of size $>m$.
  Let $\langle \bar a_i\,;b_i\,;C_i\,;\tilde C_i\ :\ i<k\rangle$ be a sequence witnessing the $\varepsilon$-instability of $\sigma(\bar x\,;z\,;X)$.
  Then for every pair $i<n$ there is some $j<m$ such that ${\sim}\varphi(a_{i,j}\,;b_n\,;\tilde C_n)$.
  By the choice of $k$ there is a subsequence of length $>m$ and a fixed $j<m$ such that ${\sim}\varphi(a_{i,j}\,;b_n\,;\tilde C_n)$ holds for every pair $i<n$ in the subsequence. This contradicts the maximality of $m$.
\end{proof}

\section{Stable definable functions}
\def\medrel#1{\parbox{5ex}{\hfil $#1$}}
\def\ceq#1#2#3{\parbox[t]{28ex}{$\displaystyle #1$}\medrel{#2}{$\displaystyle #3$}}

In this section we specialize the notions introduced in the previous sections to types $\sigma(x\,;z\,;X)$ of the form $\tau(x\,;z)\in X$, where $\tau(x\,;z)$ is an $S$-valued ${\EuScript F}$-type-definable function.

We say that $\tau(x\,;z)$ is \emph{stable\/} or \emph{$\varepsilon$-stable\/} if so is the type $\tau(x\,;z)\in X$.
The following fact is an interesting characterization of stability for functions which was first remarked in~\cite{B}.

\begin{fact}\label{fact_grothendieck}
  Let $S$ be a compact metric space.
  Let $\tau(x\,;z)$ be a term in ${\EuScript L}$.
  Then the following are equivalent
  \begin{itemize}
    \item [1.] the formula $\tau(x\,;z)\in X$ is unstable
    \item [2.] there is a sequence $\langle a_i\,;b_i\ :\ i<\omega\rangle$ such that the two limits below exist and\smallskip
    
    \noindent\kern-\labelwidth\kern-\labelsep
    \ceq{\hfill \lim_{i\to\infty}\lim_{j\to\infty}\tau(a_i\,;b_j)}{\neq}{\lim_{j\to\infty}\lim_{i\to\infty}\tau(a_i\,;b_j).}
  \end{itemize}
\end{fact}  
  Terms are a specific kind of definable functions.
  When $\tau(x\,;z)$ is a term, we have that ${\sim}\big(\tau(x\,;z)\in X\big)=\tau(x\,;z)\in X$, which is used throughout the proof below.
  We suspect that Fact~\ref{fact_grothendieck} holds more generally when $\tau(x\,;z)$ is an $S$-valued ${\EuScript F}$-type-definable function, but we have not been able to prove this.

\begin{proof}
  1$\Rightarrow$2.
  Let $C\cap\tilde C=\varnothing$ and $\langle a_i\,;b_i\ :\ i<\omega\rangle$ be as given by (1).
  That is, $\tau(a_i\,;b_j)\in C$ and $\tau(a_j\,;b_i)\in \tilde C$ hold for every $i<j<\omega$.
  We can find a subsequence such that the two limits exist; $C$ contains the limit on the left; and $\tilde C$ contains the limit on the right~--~which therefore are distinct.

  2$\Rightarrow$1. Let $C$ and $\tilde C$ be disjoint neighborhoods of the two limits in (2).
  Then (1) is witnessed by a tail of the sequence $\langle a_i\,;b_i\ :\ i<\omega\rangle$.
\end{proof}

Let $f:{\EuScript U}^{z}\to S$ be a function.
We define
{\def\medrel#1{\parbox{5ex}{\hfil $#1$}}
\def\ceq#1#2#3{\parbox[t]{15ex}{$\displaystyle #1$}\medrel{#2}{$\displaystyle #3$}}

\ceq{\hfill{\EuScript D\!}_f}{=}{\big\{\langle b,C\rangle\ \ :\ \ f(b)\in C,\ \ b\in{\EuScript U}^{z},\ \ C\in K(S)\big\}}.}

The following fact is immediate.

\begin{fact}
  The following are equivalent
  \begin{itemize}
    \item [1.] ${\EuScript D\!}_f$ is approximable by $\tau(x\,;z)\in X$
    \item [2.] for every small $B\subseteq{\EuScript U}^z$ there is an $a\in{\EuScript U}^x$ such that $\tau(a\,;b)=f(b)$ for every $b\in B$.\smallskip
  \end{itemize}
\end{fact}

We say that $f$ is \emph{approximable\/} by $\tau(x\,;z)$ if the equivalent conditions above hold.
The following is an easy consequence of Theorem~\ref{thm_stability_definability} and Theorem~\ref{thm_epsilon_stability_definability}.

\def\medrel#1{\parbox{5ex}{\hfil $#1$}}
\def\ceq#1#2#3{\parbox[t]{38ex}{$\displaystyle #1$}\medrel{#2}{$\displaystyle #3$}}

\begin{theorem}
  Let $S=[0,1]$.
  Let $\tau(x\,;z)$ be a term and let $f$ be approximable by $\tau(x\,;z)$.\smallskip
  
  1. \ If $\tau(x\,;z)$ is stable, then there are $\lambda$ and $\langle a_{i,j}\ :\ i,j<\lambda\rangle$ such that for every $b\in{\EuScript U}^z$\smallskip

  \ceq{\hfill \sup_{i<\lambda}\ \mathop{\inf\vphantom{p}}_{j<\lambda}\ \tau(a_{i,j}\,;b)}{=}{f(b).}\smallskip

  2.  If $\tau(x\,;z)$ is $\varepsilon$-stable there are $m$ and $\langle a_{i,j}\ :\ i,j<m\rangle$ such that for every $b\in{\EuScript U}^z$\smallskip

  \ceq{\hfill\Big|f(b)\ -\ \max_{i<m}\ \min_{j<m}\ \tau(a_{i,j}\,;b)\Big|}{\le}{\varepsilon.}
\end{theorem}


From the proof it is clear that the matrix $a_{i,j}$ is such that $\sup_i \inf_j$ can be replaced by $\inf_i \sup_j$.
Similarly, $\max_i \min_j$ can be replaced by $\min_i \max_j$.

\begin{proof}
  We prove the first claim.
  Let $\lambda$ and $\langle a_{i,j}\ :\ i,j<\lambda\rangle$ be as in Theorem~\ref{thm_stability_definability}.
  Let $C_i$ be the closure of the set $\{\tau(a_{i,j}\,;b)\,:\,j<\lambda\}$.
  By ($\Leftarrow$) in Theorem~\ref{thm_stability_definability}, $f(b)\in C_i$ for every $i<\lambda$.
  Then $f(b) \ge \inf_j \tau(a_{i,j}\,;b)$ for every $i<\lambda$.
  Therefore $f(b) \ge \sup_i\,\inf_j \tau(a_{i,j}\,;b)$.
 
  Now, by ($\Rightarrow$) in Theorem~\ref{thm_stability_definability} there is $k<\lambda$ such that $f(b) = \tau(a_{k,j}\,;b)$ for every $j<\lambda$.
  Therefore $f(b) = \inf_j \tau(a_{k,j}\,b)$.
  Then $f(b)\le \sup_i\inf_j \tau(a_{i,j}\,b)$.
  
  We prove the second claim.
  Let $m$ and $\langle a_{i,j}\ :\ i,j<m\rangle$ be as in Theorem~\ref{thm_epsilon_stability_definability} (without loss of generality $k=m$).
  Then Theorem~\ref{thm_epsilon_stability_definability} gives

  \ceq{3.\hfill\bigvee_{i< m}\ \bigwedge_{j<m}\ \big|\tau(a_{i,j}\,;b)-f(b)\big|\le\varepsilon.}{ }{}

  We claim that we also have

  \ceq{4.\hfill\bigwedge_{i< m}\ \bigvee_{j<m}\ \big|\tau(a_{i,j}\,;b)-f(b)\big|\ \le\ \varepsilon.}{ }{}

  For $i<m$, let $C_i=\{\tau(a_{i,j}\,;b):j<m\}$.
  Then 

  \ceq{\hfill\bigwedge_{j<m}\ \tau(a_{i,j}\,;b)\in C_i,}{ }{}

  hence by Theorem~\ref{thm_epsilon_stability_definability} (2) there is $k$ such that
  
  \ceq{\hfill\big|\tau(a_{i,k}\,;b)-f(b)\big|}{\le}{\varepsilon.}

  Then (4) follows.

  Now, suppose (2) of the theorem fails for the given $\varepsilon$.
  Then for some $b\in{\EuScript U}^z$ one of the following obtains

  \ceq{5.\hfill f(b)\ -\ \min_{j<m}\ \tau(a_{i,j}\,;b)}{>}{\varepsilon}\hfill for every $i<m$\phantom{.}

  \ceq{6.\hfill \min_{j<m}\ \tau(a_{i,j}\,;b)\ -\ f(b)}{>}{\varepsilon}\hfill for some $i<m$.

But (5) contradicts (3), while (6) contradicts (4).
\end{proof}

\vskip10ex


\BibSpec{arXiv}{%
  +{}{\PrintAuthors}{author}
  +{,}{ \textit}{title}
  +{}{ \parenthesize}{date}
  +{,}{ arXiv:}{eprint}
  +{,}{ } {note}
}

\BibSpec{webpage}{%
  +{}{\PrintAuthors} {author}
  +{,}{ \textit} {title}
  +{,}{ } {portal}
  +{}{ \parenthesize} {date}
  +{,}{ } {doi}
  +{,}{ } {note}
  +{.}{ } {transition}
}

\end{document}